\documentclass[12pt,reqno]{amsart}
\usepackage[utf8]{inputenc}
\usepackage[T1]{fontenc}
\usepackage[english]{babel}
\usepackage{xcolor}
\usepackage{hyperref}
\usepackage{appendix}

\allowdisplaybreaks

\newcommand{\R}{\mathbb{R}}

\newcommand{\eps}{\varepsilon}
\newcommand{\dm}{\mathrm{dim}}

\newtheorem{theorem}{Theorem}[section]
\newtheorem{coro}[theorem]{Corollary}
\newtheorem{proposition}[theorem]{Proposition}
\newtheorem{definition}[theorem]{Definition}
\newtheorem{lemma}[theorem]{Lemma}

\newtheorem{rem}[theorem]{Remark}

\title[Upper bounds for Steklov eigenvalues]{Upper bounds for the Steklov eigenvalues of warped products}
\author[Brisson]{Jade Brisson}
\address{Marianopolis College.
4873 Westmount Avenue, Westmount, Qu\'ebec, H3Y 1X9, Canada
}
\email{jade\textunderscore brissongg@hotmail.com}

\author[Colbois]{Bruno Colbois}
\address{Institut de Math\'ematiques, Universit\'e de Neuch\^atel, Rue Emile-Argand 11, 2000 Neuch\^atel, Suisse}
\email{bruno.colbois@unine.ch}

\author{{Alexandre} {Girouard}}\address{{D\'epartement de math\'ematiques et de statistique}, {Pavillon Alexandre-Vachon}, {Universit\'e Laval}, {Qu\'ebec QC}, {G1V 0A6}, {Canada}}
\email{alexandre.girouard@mat.ulaval.ca}

\author[Gittins]{Katie Gittins}
\address{Department of Mathematical Sciences, Durham University, Mathematical Sciences and Computer Science Building, Upper Mountjoy Campus, Stockton Road, Durham, DH1 3LE, United Kingdom.}
\email{katie.gittins@durham.ac.uk}


\date{\today}

\begin{document}
\begin{abstract}
  We obtain upper bounds for the Steklov eigenvalues of warped products $\Omega\times_h\Sigma$, where $\Omega$ is a compact Riemannian manifold with boundary and $\Sigma$ is a closed Riemannian manifold. These bounds involve the volume of $\Omega$ and of $\partial\Omega$ as well as the eigenvalues of the Laplace operator on the fiber $\Sigma$ and the $L^p$-norm of the warping function $h$. The bounds are very different depending on the dimension $n$ of the fiber $\Sigma$ and the value of $p$.
  In some cases, we obtain optimal upper bounds and stability estimates.
\end{abstract}

\maketitle
\section{Introduction}
Let $(\Omega,g_\Omega)$ be a compact, connected Riemannian manifold of dimension $d\geq 1$, with smooth boundary $\partial\Omega$ and let $(\Sigma,g_\Sigma)$ be a closed, connected Riemannian manifold of dimension $n\geq 1$. 
Given a smooth positive function $0<h\in C^\infty(\Omega)$, consider the warped product $M_h:=\Omega\times_h \Sigma$. That is the product manifold $M=\Omega\times\Sigma$ equipped with the Riemannian metric $g_h$ defined by
\[{g_h(x,q):=g_\Omega(x)+h^2(x)g_\Sigma(q)}.\]
A number $\sigma\in\R$ is a Steklov eigenvalue on $M_h$ if there exists a nonzero function $u\in C^\infty(M)$ that satisfies
$$
\begin{cases}
    \Delta u=0&\mbox{ in }M_h,\\
    \partial_\nu u=\sigma u&\mbox{ on }\partial {M_h}.
\end{cases}
$$
Here $\Delta=\Delta_{g_h}$ is the Laplace operator associated to the warped metric $g_h$ and $\partial_\nu u$ is the outward-pointing normal derivative of the function $u$.
The Steklov eigenvalues form a discrete unbounded sequence 
$$0=\sigma_0(M_h)<\sigma_1(M_h)\leq\sigma_2(M_h)\leq\cdots\nearrow+\infty.$$ 
See~\cite{LeMaPo2023} for an introduction and \cite{GiPo2017, CoGiGoSh2024} for surveys of the recent literature on the spectral geometry of the Steklov problem including upper bounds for the Steklov eigenvalues.

In order to obtain upper bounds for the Steklov eigenvalues of Riemannian manifolds, it is necessary to impose geometric constraints (see \cite{CoElGi2019, CiGi2018}). For example, in \cite{BrCoGi2024, CoGiGi2019,CoVe2021,Sel25,Tsc24} upper bounds are obtained under the assumption that the metrics are of revolution type.
In a similar spirit, warped products seem to be a next natural class to consider. Relevant papers include~\cite{DaHeNi2021,DaKaNi2021} and also \cite{GiPoly2024, Xi2021,Xi2022}, which are even closer to this project.

The initial question that motivated this project was to determine, for each fixed index $k$, how large $\sigma_k(M_h)$ can be among all warping functions $h\in C^\infty(\Omega)$ that satisfy $h\equiv 1$ on the boundary $\partial\Omega$. We succeeded in solving this problem completely for warped products with fibers $\Sigma$ of dimension $n\geq 2$. The eigenvalues of the Laplace operator on the fiber $\Sigma$ play a prominent role. We list them as
$$0=\lambda_0<\lambda_1\leq\lambda_2\leq\cdots\to+\infty.$$
The following result is a particular case of the main results of this paper.
\begin{theorem}\label{thm:initialproject}
Let $M_h=\Omega\times_h\Sigma$ be a warped product as above, with fiber $\Sigma$ of dimension $n\geq 2$.  Then,
$$\sup\{\sigma_k(M_h)\,:\,h=1 \text{ on }\partial\Omega\}=
\begin{cases}
    \frac{\lambda_k\vert \Omega\vert}{\vert \partial \Omega \vert}&\text{ if } n=2,\\
    +\infty&\text{ if }n\geq 3.
\end{cases}
$$
\end{theorem}
This could be compared with the results of~\cite{CoElGi2019,CiGi2018} where a similar question for conformal perturbations is studied.
The upper bound for $n=2$ will follow from upper bounds for $\sigma_k(M_h)$ in terms of various constraints, involving the dimension $n$ of the fiber $\Sigma$ and $L^p$-norms of the warping function $h$. The limiting behaviour is more subtle. It depends on comparison with cylindrical models  where the warping function $h$ tends to $+\infty$. Unfortunately, we have not been able to solve the case where the fiber has dimension $n=1$. See Section \ref{sec:dicussion} for a discussion.

\subsection{Main results}
Our first main result gives sharp upper bounds for $\sigma_k(M_h)$. It involves the warped spaces $M_C = \Omega \times_C \Sigma$, where the warping function $h=C\geq 1$ is constant.
\begin{theorem}\label{thm:MainIntro}
Let $M_h=\Omega\times_h\Sigma$ be a warped product as above, with fiber $\Sigma$ of dimension $n\geq 2$. Let $C\geq 1$. Then, 
\begin{enumerate}
\item[(i)] For each index $k\geq 1$, and each $0<h\in C^\infty(\Omega)$,
 \begin{equation} \label{ineq:introgeneral}
\sigma_{k}(M_h)<\lambda_k\frac{\int_{\Omega}h^{n-2}\,dV_\Omega}{\int_{\partial\Omega}h^n\,dA_{\partial\Omega}}.
\end{equation}

\item[(ii)] Moreover, if $h\leq C$ on $\Omega$ and $h=1$ on $\partial\Omega$, then for each index $k\geq 1$,
\begin{equation}\label{intro:doubleinequality}
    \sigma_k(M_h)< C^n\sigma_k(M_C)<C^{n-2}\frac{\lambda_k\vert \Omega\vert}{\vert \partial \Omega \vert}.
\end{equation}

\item[(iii)] Inequality \eqref{intro:doubleinequality} is sharp in the following sense:
\begin{equation} \label{inegsharp: M_C}
\sup\{\sigma_k(M_h)\,:\, h\le C, h=1\ \text{on} \  \partial \Omega\} =C^n \sigma_k(M_C),
\end{equation}
and 
\begin{equation} \label{Case n>2}
\lim_{C\to \infty} C^2\sigma_k(M_C)=\frac{\lambda_k\vert \Omega\vert}{\vert \partial \Omega \vert} .
\end{equation}
\end{enumerate}
\end{theorem}
Theorem~\ref{thm:initialproject} is an immediate consequence of inequality \eqref{intro:doubleinequality} and of the asymptotics given in \eqref{inegsharp: M_C} and in \eqref{Case n>2}.
The proof of Theorem~\ref{thm:MainIntro} uses separation of variables which gives the decomposition of the Steklov eigenvalues of $M_h$ as a union of the spectra of auxiliary Steklov problems for diffusion operators $-\Delta-\frac{n}{h}\nabla h+\lambda_kh^{-2}$ on $\Omega$. Once this is available, the proof of the upper bound in part (i) is a straightforward application of the variational characterisation of the eigenvalues. However, the main difficulty resides in the proof of the saturation results in part (iii). 

We record the case $n=2$ separately since it will play a special role.
\begin{coro}
Let $M_h=\Omega\times_h\Sigma$ be a warped product as above, with fiber $\Sigma$ of dimension $n=2$. Let $C\geq 1$.
If $h\leq C$ on $\Omega$ and $h=1$ on $\partial \Omega$, then for each $k\geq 1$,
\begin{equation} \label{Case n=2}
\sigma_k(M_h)< C^2\sigma_k(M_C)<\frac{\lambda_k\vert \Omega\vert}{\vert \partial \Omega \vert},
\end{equation}
and
\begin{equation} \label{Case n=2 sharp}
\sup\{\sigma_k(M_h), h=1\ \text{on} \  \partial \Omega\} =\frac{\lambda_k\vert \Omega\vert}{\vert \partial \Omega \vert} .
\end{equation}
\end{coro}

In our construction, \eqref{Case n=2 sharp} is attained by a sequence of warping functions that become very large in the interior of $\Omega$. The next result shows that this condition is actually required in order to approach the optimal bound.
Indeed, by using  a more precise trial function we obtain the following quantitative stability improvement of \eqref{Case n=2}.
\begin{theorem} \label{thm:Stability}
Let $n=2$. For each index $k\geq 1$, if $h_\eps$ is a family of warping functions such that $h_\eps\equiv 1$ on $\partial\Omega$ and if
$$\sigma_k(M_{h_\eps})\xrightarrow{\eps\to0}\frac{\lambda_k|\Omega|}{|\partial\Omega|},$$
then for each subdomain $D\subset\Omega$, 
\begin{equation}\label{eq:limstab}
    \lim_{\eps\to 0}\int_{D}h_\eps^2\,dV_{\Omega}=+\infty.
\end{equation}
More precisely, for any warping function $h$ such that $h \equiv 1$ on $\partial \Omega$, let $\delta(h):= \frac{\lambda_k \vert \Omega\vert}{\vert \partial \Omega\vert}-\sigma_k(M_{h})>0$. If $q\in \Omega$ and $B(q,r) \subset \Omega$ is the ball of radius $r$ and centre $q$, then the following quantitative estimate holds:
$$
\int_{B(q,r)}h^2dV_{\Omega}\ge \frac{\lambda_kr^2}{4\delta(h)}\frac{\vert B(q,\frac{r}{2})\vert}{\vert\partial \Omega\vert}-\lambda_kr^2 \vert B(q,r)\vert.
$$    
\end{theorem}

\begin{rem}
    In view of Theorem~\ref{thm:Stability} it would be interesting to investigate whether similar quantitative stability improvements are available for fibers $\Sigma$ of arbitrary dimension $n$. 
\end{rem}

For $h=1$ on $\partial\Omega$, inequality~\eqref{ineq:introgeneral} provides an upper bound in terms of the $L^{n-2}$-norm of $h$ while the hypothesis $h\leq C$ is an $L^\infty$-norm constraint. The H\"older inequality leads to the following immediate corollary.

\begin{coro}\label{coro: integralbound}
Let $M_h=\Omega\times_h\Sigma$ be as above, with fiber $\Sigma$ of dimension $n\geq 3$. Suppose that $h=1$ on $\partial \Omega$.
If $1\leq n-2\leq p<\infty$ then, for each index $k\geq 1$,
\begin{equation} \label{ineg: n ge 3}
\sigma_k(M_h) < \frac{\lambda_k \vert \Omega \vert^{\frac{p-(n-2)}{p}}}{|\partial\Omega|}\|h\|_{L^p}^{n-2}.
\end{equation}       
\end{coro}
\begin{rem}
    One advantage of the $L^p$ bound is that it could remain bounded even if the $L^\infty$-norm blows up to infinity.
\end{rem}
\begin{rem}
Because $h=1$ on $\partial\Omega$, for $p=n-2$, we see that inequality \eqref{ineg: n ge 3} is a special case of inequality \eqref{ineq:introgeneral}. On the other hand, letting $p\to\infty$ in \eqref{ineg: n ge 3} leads to 
$$\sigma_k(M_h) \leq \frac{\lambda_k \vert \Omega \vert}{|\partial\Omega|}\|h\|_{L^\infty}^{n-2},$$
which is the non-strict version of inequality \eqref{intro:doubleinequality}.
\end{rem}
This naturally raises the question to know what happens for $1\leq p< n-2$.
\begin{theorem}\label{thm:3dimconnected}
Let $M_h=\Omega\times_h\Sigma$ be as above, with fiber $\Sigma$ of dimension $n\geq 3$. Suppose that $h$ takes the value $1$ on $\partial \Omega$.
If $1\leq p<n-2$ and if $\partial \Omega$ is connected, then for each positive constant $C>0$,
\begin{equation}  \label{case large n}
\sup\{\sigma_k(M_h) : \int_{\Omega}h^pdV_{\Omega}\leq C; h=1\ \text{on}\ \partial \Omega \}=+\infty.
\end{equation}
\end{theorem}
In our proof of Theorem \ref{thm:3dimconnected}, the connectedness of the boundary plays an important role. We believe the supremum \eqref{case large n} is still infinite for general boundary, if the dimension $d$ of the base $\Omega$ is at least 2. 

In contrast, if the base has dimension $d=1$, then the corresponding supremum is bounded for all $p\geq 1$, even though the boundary $\partial\Omega=\{0,L\}$ of an interval $\Omega=[0,L]$ is not connected.
\begin{theorem}\label{dim basis 1}
Let $M_h=[0,L]\times_h\Sigma$, with fiber $\Sigma$ of dimension $n\geq 2$. Let $p\ge 1$. 
Suppose that $h(0)=h(L)=1$.
Then, for $k \geq 1$,
\begin{equation}\label{ineq:1dkgen}
    \sigma_k(M_h) \leq \frac{3^{n-2}}{4}\|h\|_{L^p}^{n-2}L^{\frac{p-(n-2)}{p}}\lambda_k.
\end{equation}
Moreover, for $k=1$ we also have
\begin{equation}\label{ineq:1dk1}
\sigma_1(M_h)\le \frac{4 \cdot 3^n \|h\|_{L^p}^{n}}{L^{\frac{n+p}{p}}}.
\end{equation}
\end{theorem}
\begin{rem}
By introducing the normalised measure $d\tau=dt/L$, inequality
\eqref{ineq:1dkgen} can be restated as follows:
\begin{equation}
    \sigma_k(M_h) \leq \frac{3^{n-2}}{4}\|h\|_{L^p(d\tau)}^{n-2}L\lambda_k.
\end{equation}
While for $k=1$ inequality \eqref{ineq:1dk1} becomes
\begin{equation}
\sigma_1(M_h)\le \frac{4 \cdot 3^n \|h\|_{L^p(d\tau)}^{n}}{L}.
\end{equation}
These inequalities can be directly applied to constant warping function $h=C$, leading to 
$$\lim_{L\to 0}\sigma_k(M_C)=0\qquad \text{ for each } k$$
as well as 
$$\lim_{L\to\infty}\sigma_1(M_C)=0.$$
This captures the familiar situation for cylinders (see, for example, \cite[Example 2.2]{CoGiGoSh2024}.)
\end{rem}

\subsection{One-dimensional fibers}\label{sec:dicussion}
Despite our best efforts, we were not able to solve our initial problem for warped products with fibers $\Sigma$ of dimension $n=1$. In that case, inequality \eqref{ineq:introgeneral} is still true, and when $h=1$ on $\partial\Omega$ it reads as follows:
\begin{equation*}
\sigma_{k}(M_h)<
\lambda_k\frac{\int_{\Omega}h^{-1}\,dV_\Omega}{|\partial\Omega|}.
\end{equation*}
The strategy that we have used to produce large eigenvalues when $n\geq 2$ is counterproductive when $n=1$. Indeed, for $n=1$, if $h\to +\infty$ in the interior of $\Omega$, then $\sigma_k(M_h)\to 0$.
The difficulty of the problem will become clear once we have presented the separation of variables for warped products and the min-max characterisation of the auxiliary separated problems. See the corresponding Rayleigh quotient \eqref{RayleighQuotient} and note that for $n=1$, the warping factor $h$ appears with both positive and negative exponents.

\begin{rem} Even in the situation when both $d=1$ and $n=1$, the problem is quite complicated. Consider the cylinder $[0,L]\times S^1$ with the warped metric $dr^2+h(r)^2d\theta^2$. The unknown change of variable $r=r(t)$ leads to the metric $r'(t)^2dt^2+h(r)^2d\theta^2$. In order to obtain a metric that is conformal to the product metric $dr^2+d\theta^2$ we solve $r'(t)=h(r)$. For $r\in[0,L]$, let
    $$t(r):=\int_0^r\frac{1}{h(s)}\,ds.$$
    The change of variable $t\mapsto r(t)$ is the inverse of that function.
    It follows that $[0,L]\times_h S^1$ is isometric to $[0,t(L)]\times S^1$ equipped with the conformal Riemannian metric $h(r(t))^2(dt^2+d\theta^2)$. Because the conformal factor restricts to $h=1$ on the boundary, it follows from conformal invariance that 
    $$\sigma_k([0,L]\times_hS^1)=\sigma_k([0,t(L)]\times S^1)\qquad\text{ for each } k.$$
    Hence, the supremum of $\sigma_k([0,L]\times_h S^1)$ among all warping functions $h$ such that $h(0)=h(L)=1$ is equal to
    $$\sup\{\sigma_k([0,T]\times S^1)\,:\,T>0\}.$$
    This last problem has been studied by Fan, Tam and Yu in their paper~\cite{FaTaYu2015}. In particular, Theorem 1.1 of their paper gives a complete answer to this maximisation problem. The interested reader is also invited to look at~\cite[Section 5.7]{CoGiGoSh2024}.
\end{rem}

\section*{Plan of the paper}
In Section \ref{sec:prelim} we lay the foundations for the proofs of the main theorems by recalling the separation of variables on warped products, investigating the spectrum of the particular case $M_C$, presenting some technical lemmas, and recalling some facts about quasi-isometries. We defer the proofs of Theorems \ref{thm:MainIntro}, \ref{thm:Stability}, \ref{thm:3dimconnected}, \ref{dim basis 1} and Corollary \ref{coro: integralbound} to Section \ref{proof: theorems}.

\section*{Acknowledgements}
J.B and B.C acknowledge support of the SNSF project \emph{``Geometric Spectral Theory''}, grant number 200021-19689.
K.G. acknowledges support from a Royal Society Research Grant entitled \emph{``Upper bounds for Steklov eigenvalues, ratios and gaps''}. A.G. acknowledges the support of the Natural Sciences and Engineering Research Council of Canada (NSERC), in particular through the discovery grant \emph{``Isoperimetry and spectral geometry''} (RGPIN-2022-04247) as well as support of Fonds de recherche du Québec (FRQ) through the team research project \emph{``G\'eom\'etrie des \'ecarts spectraux pour les probl\`emes de Laplace et de Steklov''}.

\section{Preliminary material}\label{sec:prelim}

In this section, we first recall the separation of variables on warped products and list some of its consequences.
We then introduce some technical lemmas and recall some facts about quasi-isometries that will be used in the proof of Theorem \ref{thm:MainIntro}. Finally, we obtain some properties of the spectrum for the example of $M_C = \Omega \times_{C} \Sigma$ which will also be useful in the proof of Theorem \ref{thm:MainIntro}.

\subsection{Separation of variables on warped products}\label{sec:sepvar}
The warped product structure allows the use of separation of variables. This is expressed in terms of auxiliary Steklov problems for the diffusion operator $L_{h}:C^\infty(\Omega)\to C^\infty(\Omega)$ on the base, which is defined by
$$L_{h}a=-\Delta a-\frac{n}{h}\langle \nabla h,\nabla a\rangle.$$
For each $\lambda\geq 0$, consider the following auxiliary Steklov problem:
\begin{gather}\label{problem:auxiliary}
\begin{cases}
    L_ha+\lambda h^{-2} a=0&\mbox{ in }\Omega,\\
    \partial_\nu a=\sigma a&\mbox{ on }\partial \Omega.
\end{cases}
\end{gather}
If the dimension of the base $\Omega$ is $d\geq 2$, then for each $\lambda\geq 0$, this problem has discrete unbounded spectrum:
\[0\leq\sigma_{\lambda,0}(h)\leq\sigma_{\lambda,1}(h)\leq\sigma_{\lambda,2}(h)\leq\ldots\nearrow+\infty.\]
Corresponding eigenfunctions $(a_{\lambda,l})_{l\geq 0  }$ form a complete orthogonal set 
that we normalise by $\int_{\partial\Omega}a_{\lambda,l}^2h^n\,dA=1$.
However, if the base is 1-dimensional, there are only two eigenvalues for each $\lambda$: 
$$\{\sigma_{\lambda,0}(h)\leq\sigma_{\lambda,1}(h)\}.$$

\begin{proposition}\label{prop:separation}
Let the eigenvalues of the Laplace operator on the fiber $\Sigma$ be $0=\lambda_0<\lambda_1\leq\lambda_2\leq\cdots\nearrow+\infty$.
The Steklov spectrum of $M_h=\Omega\times_h\Sigma$ is given by the multiset
\begin{gather}\label{eq:separatedeigenvalues}
\{\sigma_k(M_h)\}_{k\geq 0}=\{\sigma_{\lambda_j,l}(h)\,:\,j,l\geq 0\}.
\end{gather}    
\end{proposition}
The proof is standard: it follows by considering products of functions of the form 
$$u_{\lambda_j,l}(x,p)=a_{\lambda_j,l}(x)\phi_j(p),$$ 
where the functions $a_{j,l}\in C^\infty(\Omega)$ are eigenfunctions of the above auxiliary Steklov problem~\eqref{problem:auxiliary} for $\lambda=\lambda_j$ and where the functions $\phi_j\in C^\infty(\Sigma)$ are Laplace eigenfunctions
forming an orthonormal basis of $L^2(\Sigma)$. These functions form a complete set of Steklov eigenfunctions on $M_h$. For details, see~\cite{GiPoly2024}. 
Of particular interest to us is that the eigenvalues $\sigma_{\lambda,l}$ admit familiar variational characterisations in terms of the following Rayleigh quotient for $a\in W^{1,2}(\Omega)$:
\begin{gather}\label{RayleighQuotient}
R_{\lambda,h}(a)=\frac{\int\limits_\Omega (|da|^2h^n+\lambda a^2h^{n-2})\,dV_\Omega}{\int\limits_{\partial\Omega}a^2h^n\,dA_{\partial\Omega}}\,.    
\end{gather}
We will make extensive use of the following characterisation of $\sigma_{\lambda,0}$:
\begin{gather}\label{eq:cara3}
    \sigma_{\lambda,0}(h)=\min\bigg\{R_{\lambda,h}(a):a\in W^{1,2}(\Omega)\setminus\{0\}\bigg\}=R_{\lambda,h}(a_{\lambda,0}),
\end{gather}
with equality if and only if $a$ is a corresponding eigenfunction.

Observe that for each $k\geq 0$,
\begin{equation}\label{ineq:kcompk0}
    \sigma_k(M_h)\leq \sigma_{\lambda_k,0}(h).
\end{equation}
This follows directly from the representation~\eqref{eq:separatedeigenvalues}, since the labelled eigenvalues $\sigma_k(M_h)$ are listed in increasing order by definition. 
By using the constant function $a\equiv 1$ as a trial function in~\eqref{eq:cara3}, this leads to
\begin{gather}\label{ineq:consttest}
\sigma_k(M_h)\leq \sigma_{\lambda_k,0}(h)\leq\lambda_k\frac{\int_{\Omega}h^{n-2}\,dV_\Omega}{\int_{\partial\Omega}h^n\,dA_{\partial\Omega}}.
\end{gather}

\noindent
For arbitrary $\lambda\geq 0$ and index $l\geq 1$, we will also use the characterisation
\begin{gather}
\sigma_{\lambda,l}(h)= \min_{E\subset W^{1,2}(\Omega), \dm (E)=l+1} \max\{R_{\lambda,h}(u):\ u \in E\}.
\end{gather}

 \subsection{Steklov spectrum of $M_C$} \label{Spectrum $M_C$} Recall that for $C>0$,
$$M_C=\Omega \times_C \Sigma = (\Omega,g_{\Omega})\times (\Sigma,C^2g_{\Sigma}).$$
In this section we will investigate the spectrum of $M_C$ when $C$ is large. This will be useful for the proof of Theorem \ref{thm:MainIntro}. In particular, we will obtain inequality \eqref{Case n>2} shortly.

For $h=C>0$, problem \eqref{problem:auxiliary} simplifies to
\begin{gather}\label{eq:sepwarpedconstant}
\begin{cases}
    -\Delta a+\frac{\lambda}{C^2}a=0&\mbox{ in }\Omega,\\
    \partial_\nu a=\sigma a&\mbox{ on }\partial \Omega.
\end{cases}
\end{gather}
This is the Steklov problem for the Helmholtz equation with energy $-\mu=-\lambda/C^2$.
The smallest eigenvalue of this equation is given by
$$\sigma_{\mu,0}=\min\left\{\int_{\Omega}\bigl(|\nabla f|^2+\mu f^2\bigr)\,dV_\Omega\,:\,f:\Omega\to\R, \int_{\partial\Omega}f^2\,dA_{\partial\Omega}=1\right\}.$$
Because $\sigma_{\mu,0}$ is the smallest eigenvalue, it is simple. Let $f_\mu$ be the corresponding eigenfunction. Then we can compute the derivative with respect to $\mu$,
$$\sigma_{\mu,0}'=2\int_{\Omega}\bigl(\nabla f_\mu\cdot \nabla f_\mu'+\mu f_\mu f_\mu'\bigr)\,dV_\Omega+\int_\Omega f_\mu^2\,dV_\Omega.$$
We claim that the first integral on the right-hand side vanishes. Indeed, the function $f_\mu$ satisfies the weak equation corresponding to \eqref{eq:sepwarpedconstant}:
$$\int_\Omega \bigl(\nabla f_\mu\cdot \nabla v+\mu f_\mu v\bigr)\,dV_\Omega=\sigma\int_{\partial\Omega}f_\mu v\,dA_{\partial\Omega}\qquad \forall v\in C^\infty(\Omega).$$
Using $v=f_\mu'$ as a test function gives
$$\int_\Omega \bigl(\nabla f_\mu \cdot \nabla f_\mu'+\mu f_\mu f_\mu'\bigr)\,dV_\Omega=\sigma\int_{\partial\Omega}f_\mu f_\mu'\,dA_{\partial\Omega}.$$
Now, differentiating the normalisation $\int_{\partial\Omega}f_\mu^2=1$ with respect to $\mu$ gives $\int_{\partial\Omega}f_\mu f_\mu'\,dA_{\partial\Omega}=0$, proving the claim.
It follows that
$$\sigma_{\mu,0}'=\int_\Omega f_\mu^2\,dV_\Omega=\frac{\int_\Omega f_\mu^2\,dV_\Omega}{\int_{\partial\Omega} f_\mu^2\,dA_\Omega}.$$
In particular, if we compute the derivative at $\mu=0$, using that $f_0$ is constant we get
$$\frac{d}{d\mu}\sigma_{\mu,0}\Bigl|\Bigr._{\mu=0}=\frac{|\Omega|}{|\partial\Omega|}.$$
Now in order to prove \eqref{Case n>2} of Theorem \ref{thm:MainIntro}, we compute:
\begin{gather*}
    \lim_{C\to\infty}C^2\sigma_{\lambda_k/C^2,0}
    =\lim_{C\to\infty}\lambda_k\frac{\sigma_{\lambda_k/C^2,0}-\sigma_{0,0}}{\lambda_k/C^2}
    =\lambda_k\sigma_{0,0}'
    =\lambda_k|\Omega|/|\partial\Omega|.
\end{gather*}
Finally, observe that for  $\lambda=\lambda_k$ with $k\ge 1$, taking the constant function $a=1$ as a trial function leads to
\begin{equation} \label{ineg: C}
    \sigma_{\lambda_k,0}(C) \le \frac{\lambda_k \vert \Omega\vert}{C^2\vert \partial \Omega \vert}.
\end{equation}
\noindent
On the other hand, $\sigma_{\lambda_k,1}(C)\ge \sigma_{0,1}(C)=\sigma_1(\Omega).$
We conclude that, if $C$ is large enough then the first $k+1$ eigenvalues of $M_C$ are precisely $\sigma_{\lambda_j,0}(M_C)$, $j=0,1,\dots,k$.
In particular, 
    $\lim_{C\to\infty}C^2\sigma_{k}(M_C)=\lim_{C\to\infty}C^2\sigma_{\lambda_k/C^2,0}$, which concludes the proof of 
    \eqref{Case n>2} of Theorem \ref{thm:MainIntro}.

\subsection{Almost orthogonality}\label{sec:teclem}
Let $N$ be a closed Riemannian manifold. Given a smooth function $f\in C^\infty([0,\eps]\times N)$, for each $t\in [0,\eps]$ we define 
$$f_t:=f|_{\{t\}\times N}.$$ 
 If a sequence of functions $(f_j)\in C^\infty([0,\eps]\times N)$ is such that $(f_j)_0$ form an orthonormal family in $L^2(N\times\{0\})$, then the following two lemmas will be used to show that $(f_j)_t$ will be almost orthonormal, with defect controlled in terms of the Dirichlet energy of the functions.
These results will be used in the proof of inequality \eqref{inegsharp: M_C} and also in the proof of Theorem \ref{thm:3dimconnected}.
\begin{lemma}\label{lemma:deBruno}
 Let $N$ be a closed Riemannian manifold. Let $f\in C^\infty([0,\eps]\times N)$. For $t\in [0,\eps]$,  define $f_t:=f|_{\{t\}\times N}$. The following holds for all $t\in [0,\eps]$:
 \begin{gather*}
   \bigl|\|f_\eps\|_{L^2(N)}-\|f_t\|_{L^2(N)}\bigr|^2
   \leq
   (\eps-t)\int_{[t,\eps]\times N}|d f|^2\,dV_N.
 \end{gather*}
\end{lemma}
\begin{proof}
    It follows from the Fundamental Theorem of Calculus and the Cauchy--Schwarz inequality that
    \begin{align*}
        (f_\eps(x)-f_t(x))^2
        =
        \left(\int_t^\eps\frac{\partial f}{\partial s}(s,x)ds\right)^2
        \leq
        (\eps-t)\int_t^\eps\bigg|\frac{\partial f}{\partial s}(s,x)\bigg|^2ds.
    \end{align*}
    Integrating over $N$   
    leads to
    \begin{equation}\label{eq:lem2.1}
    \|f_\eps-f_t\|^2_{L^2(N)}
    \leq
    (\eps-t) \int_{[t,\eps]\times N}\bigg|\frac{\partial f}{\partial s}\bigg|^2
    \leq
    (\eps-t) \int_{[t,\eps]\times N} |df(s,x)|^2.
    \end{equation}
    The proof is completed by using the triangle inequality:
    $$\bigl|\|f_\eps\|_{L^2(N)}-\|f_t\|_{L^2(N)}\bigr|^2
     \leq
    (\eps-t)\int_{[t,\eps]\times N} |df(s,x)|^2.
    $$
\end{proof}

There is a comparable result for the product of two functions.

\begin{lemma}  \label{lemma:deBrunobis} 
Let $N$ be a closed Riemannian manifold.
Let $f,\tilde{f} \in C^\infty([0,\eps]\times N)$. 
For $t\in [0,\eps]$,  define $f_t:=f|_{\{t\}\times N}$ and $\tilde{f}_t:=\tilde{f}|_{\{t\}\times N}$.
Then for $t < \eps$,
\begin{align*}
&\bigg\vert \int_N f_{\eps}(x)\tilde{f}_{\eps}(x)dx-\int_N f_{t}(x)\tilde{f}_{t}(x)dx \bigg\vert  \\
&\qquad \qquad \le \vert \eps-t\vert^{1/2} \left(\left(\int_{[t,\eps]\times N} |df|^{2} \right)^{1/2}\left(\int_{N} \tilde{f}_{\eps}^{2}\right)^{1/2} \right. \nonumber\\
& \qquad \qquad \qquad \qquad \qquad \qquad \left. + \left(\int_{[t,\eps]\times N}|d\tilde{f}|^{2} \right)^{1/2}\left(\int_{N} f_{t}^{2}\right)^{1/2}  \right).
\end{align*}
\end{lemma}

\begin{proof} 
We have
$$
f_{\eps}(x)\tilde{f}_{\eps}(x)-f_t(x)\tilde{f}_t(x)=
(f_{\eps}(x)-f_t(x))\tilde{f}_{\eps}(x)+(\tilde{f}_{\eps}(x)-\tilde{f}_t(x))f_{t}(x).
$$
As in the proof of Lemma \ref{lemma:deBruno}, we see that
\begin{align*}
&\vert f_\eps(x)-f_t(x) \vert
         \leq
        \vert \eps-t \vert^{1/2}\left(\int_t^\eps\bigg|\frac{\partial f}{\partial s}(s,x)\bigg|^2ds\right)^{1/2},\\
&\vert \tilde{f}_\eps(x)-\tilde{f}_t(x) \vert
         \leq
        \vert \eps-t \vert^{1/2}\left(\int_t^\eps\bigg|\frac{\partial \tilde{f}}{\partial s}(s,x)\bigg|^2ds\right)^{1/2}.
\end{align*}

Integrating over $N$, we obtain
\begin{align*}
&\int_N \vert (f_{\eps}(x)-f_t(x))\tilde{f}_{\eps}(x)\vert\\
& \qquad \qquad \le \left(\int_N \vert f_{\eps}(x)-f_t(x)\vert^2\right)^{1/2}\left(\int_N \tilde{f}_{\eps}^2(x)\right)^{1/2}\\
&\qquad \qquad \le \vert \eps-t\vert^{1/2}\left(\int_N \int_t^{\eps} \bigg\vert \frac{\partial f}{\partial s}(s,x)\bigg\vert^2 \right)^{1/2} \left(\int_N \tilde{f}_{\eps}^2(x)\right)^{1/2} \\
&\qquad \qquad \le\vert \eps-t\vert^{1/2}\left(\int_{[t,\eps]\times N} \vert df\vert^2\right)^{1/2}\left(\int_N \tilde{f}_{\eps}^2(x)\right)^{1/2},
\end{align*}
and in the same way
$$
\int_N \vert (\tilde{f}_{\eps}(x)-\tilde{f}_t(x))f_{t}(x)\vert \le
\vert \eps-t\vert^{1/2}\left(\int_{[t,\eps]\times N} \vert d\tilde{f}\vert^2\right)^{1/2}\left(\int_N f_{t}^2(x)\right)^{1/2}.
$$
\end{proof}

\subsection{Steklov eigenvalues and quasi-isometries}
In this section we recall the definition of a quasi-isometry and results for Steklov eigenvalues under quasi-isometries as these concepts will be useful in the proofs of \eqref{inegsharp: M_C} and \eqref{case large n}. The results are exactly those of \cite{Br2022}.

\begin{definition}
Two Riemanian metrics $g_1$ and $g_2$ on a manifold $M$ are \emph{quasi-isometric with ratio $K \geq 1$} if for all $q \in M$ and for all non-zero $v\in T_qM$,
$$\frac{1}{K} \leq \frac{g_1(v,v)}{g_2(v,v)} \leq K.$$
More generally, a diffeomorphism $\phi:(M_1,g_1)\to (M_2,g_2)$ between two Riemannian manifolds is a quasi-isometry with ratio $K$ if $\phi^\star g_2$ is quasi-isometric to $g_1$ with ratio $K$.
\end{definition}
The following result is \cite[Proposition 2.4]{Br2022}.
\begin{proposition}
    Let $M$ be a Riemannian manifold of dimension $m$. If $g_1$ and $g_2$ are two Riemannian metrics on $M$ that are quasi-isometric with ratio $K$, then the Steklov eigenvalues with respect to $g_1$ and to $g_2$ satisfy
    $$ \frac{1}{K^{m+1/2}} \leq \frac{\sigma_k(M,g_1)}{\sigma_k(M, g_2)} \leq K^{m +1/2}.$$
\end{proposition}

The following lemma will enable us to employ Lemma \ref{lemma:deBruno} and Lemma \ref{lemma:deBrunobis} in the proofs of \eqref{inegsharp: M_C} and \eqref{case large n} that follow.
\begin{lemma}\label{lemma:quasi-iso}
Given $\rho>0$, it is classical to construct a new Riemannian metric $g_{\rho}$ on $\Omega$ with the following properties:
\begin{itemize}
    \item $g_{\rho}=g_{\Omega}$ on $\partial \Omega$,
    \item $g_{\rho}$ is quasi-isometric to the initial metric $g_{\Omega}$ with ratio $1+\rho$,
    \item there is $\delta=\delta(\rho) \to 0$ as $\rho \to 0$, such that in the neighbourhood $T_{\delta}=\{q\in \Omega: d_{g_{\Omega}}(q,\partial \Omega)< \delta\}$, $(T_{\delta},g_{\rho})$ is isometric to the product $(\partial \Omega,g_{\rho})\times [0,\delta]$. 
\end{itemize}
\end{lemma}
This lemma follows by using \cite[Proposition 2.2]{Br2022} to make a neighbourhood of the boundary $\partial \Omega$ isometric to a product, and then proceeding exactly as in the proof of Theorem 1.1, p. 2822-23 in \cite{CoElGi2019} to extend the metric to $\Omega$.

\section{Proofs of the theorems} \label{proof: theorems}

In this section, we prove the main theorems. 

\subsection{Proof of Theorem \ref{thm:MainIntro}} The proof of Theorem \ref{thm:MainIntro} goes, roughly speaking, as follows: in general, it will be straightforward to find upper bounds for the eigenvalues, but difficult to show that they are sharp, when they are.

\begin{proof}[Proof of Inequality (\ref{ineq:introgeneral}).]
In~\eqref{ineq:consttest} we have already obtained that 
$$\sigma_k(M_h)\leq \lambda_k\frac{\int_{\Omega}h^{n-2}\,dV_\Omega}{\int_{\partial\Omega}h^n\,dA_{\partial\Omega}}.$$
The inequality has to be strict. Indeed, if $\sigma_{\lambda_k,0}(h)=\lambda_k\frac{\int_{\Omega}h^{n-2}\,dV_\Omega}{\int_{\partial\Omega}h^n\,dA_{\partial\Omega}}$, then $a\equiv 1$ must be an eigenfunction corresponding to $\sigma_{\lambda_k,0}(h)$. However, the strong form of the eigenvalue equation gives:
$0=L_ha+\lambda_k h^{-2}a=\lambda_k h^{-2} a.$ From which it follows that $\lambda_k=0$, contradicting the hypothesis.
\end{proof}

\begin{proof}[Proof of Inequality \eqref{intro:doubleinequality}.]
We note that, from Inequality \eqref{ineq:introgeneral}, using that $h \leq C$ on $\Omega$ and $h=1$ on $\partial \Omega$, we obtain
  $$\sigma_k(M_h) < C^{n-2}\frac{\lambda_k \vert \Omega \vert}{\vert \partial \Omega \vert}.$$
  
To prove Inequality \eqref{intro:doubleinequality}, for $a\in W^{1,2}(\Omega)$ and fixed $\lambda > 0$, since $h=1$ on $\partial\Omega$, we have
\begin{align*} 
R_{\lambda,h}(a)&=\frac{\int\limits_\Omega (|da|^2h^n+\lambda a^2h^{n-2})\,dV_\Omega}{\int\limits_{\partial\Omega}a^2\,dA_{\partial\Omega}} \\
& \ \ \ \le \frac{\int\limits_\Omega (|da|^2C^n+\lambda a^2C^{n-2})\,dV_\Omega}{\int\limits_{\partial\Omega}a^2\,dA_{\partial\Omega}}=C^n\frac{\int\limits_\Omega (|da|^2+\frac{\lambda}{C^2} a^2 )\,dV_\Omega}{\int\limits_{\partial\Omega}a^2\,dA_{\partial\Omega}},
\end{align*}
and the expression
$$
\frac{\int\limits_\Omega (|da|^2+\frac{\lambda}{C^2} a^2 )\,dV_\Omega}{\int\limits_{\partial\Omega}a^2\,dA_{\partial\Omega}}
$$
  is the Rayleigh quotient for the product metric $g_C=g_{\Omega} \oplus C^2 g_{\Sigma}$. If $\lambda_k$ is the $k$-th eigenvalue of $(\Sigma,g_{\Sigma})$, we note that $\frac{\lambda_k}{C^2}$ is the $k$-th eigenvalue of $(\Sigma,C^2g_{\Sigma})$.
  Hence, by \eqref{ineq:kcompk0},
  \begin{align}
      \sigma_k(M_h) &\leq \sigma_{\lambda_k,0}(h) \nonumber\\
  &= \min\bigg\{R_{\lambda_k,h}(a):a\in W^{1,2}(\Omega)\setminus\{0\}\bigg\} \nonumber\\
  & \leq \min \bigg\{ C^n\frac{\int\limits_\Omega (|da|^2+\frac{\lambda_k}{C^2} a^2 )\,dV_\Omega}{\int\limits_{\partial\Omega}a^2\,dA_{\partial\Omega}}: a\in W^{1,2}(\Omega)\setminus\{0\}  \bigg\} \label{RQmc}\\
  &= C^n \sigma_k(M_C). \nonumber
  \end{align}

  The inequality has to be strict: in case of equality, there is an eigenfunction $a$ with
  $$\frac{\int\limits_\Omega (|da|^2h^n+\lambda_k a^2h^{n-2})\,dV_\Omega}{\int\limits_{\partial\Omega}a^2\,dA_{\partial\Omega}} = \frac{\int\limits_\Omega (|da|^2C^n+\lambda_k a^2C^{n-2})\,dV_\Omega}{\int\limits_{\partial\Omega}a^2\,dA_{\partial\Omega}},$$
  so
  $$ |da|^2(h^n - C^n) + \lambda_k a^2(h^{n-2} - C^{n-2}) = 0 \text{ almost everywhere.}$$
  Since $h<C$ in a neighbourhood of $\partial \Omega$, $a$ has to be equal to $0$ in this neighbourhood, which is not possible.

  Taking $a=1$ as a test function in \eqref{RQmc} gives
  $$C^n\sigma_k(M_C) \leq C^{n-2}\frac{\lambda_k \vert \Omega \vert}{\vert \partial \Omega \vert}.$$
  The fact that the inequality has to be strict follows by the same argument as in the proof of Inequality (\ref{ineq:introgeneral}).
  \end{proof}

  \begin{proof}[Proof of statement (\ref{inegsharp: M_C}).] 
  By Inequality \eqref{intro:doubleinequality}, we already have that 
  $$\sup\{\sigma_k(M_h), h\le C, h=1\ \text{on} \  \partial \Omega\} \le C^n \sigma_k(M_C).$$ 
  But it is more difficult to show that the supremum is equal to $C^n \sigma_k(M_C)$.

  The idea itself to show this fact is simple: we construct a family of metrics $g(h_t)=g_{\Omega}\oplus h_t^2g_{\Sigma}$, $0<t<1$ where $h_t$ interpolates between the value $1$ on $\partial \Omega$ and the value $C$ outside a small neighbourhood of width $t$ of $\partial \Omega$. One may hope that the spectrum of $g(h_t)$ will converge to $C^n \sigma_k(M_C)$ as $t\to 0$. The difficulty is that the eigenvalues $\sigma_k(M_{h_t})$ depend on $t$, which makes the comparison more delicate.

We make a construction in three steps:
\begin{enumerate}
    \item[(i)]  Firstly, we replace the metric $g_{\Omega}$ on $\Omega$ by a new metric $g_{\rho}$ which is very close to $g_{\Omega}$ and is quasi-isometric to a product $\partial \Omega \times I$ in a neighbourhood of $\partial \Omega$, where $I$ is a small interval.
    \item[(ii)] Next, we introduce a domain $\Omega_{\delta}$ which is strictly contained in $\Omega$ and is quasi-isometric to $(\Omega,g_{\rho})$. 
    \item[(iii)] We then construct a family of warping functions which take the value $1$ on $\partial \Omega$ and the value $C$ on $\Omega_{\delta}$. This allows us to compare the eigenvalue $\sigma_{k}(M_h)$ with the eigenvalue $C^n\sigma_k(M_C)$, by using Lemma \ref{lemma:deBrunobis}.
\end{enumerate}

\noindent
Firstly, Lemma \ref{lemma:quasi-iso} gives (i).
Next, to achieve (ii), we
  introduce the domain
  $\Omega_{\delta} \subset \Omega$, where
  $$\Omega_{\delta}=\{p\in \Omega:d_{g_{\rho} }(p,\partial \Omega)\ge \delta^2\}.$$ 
  The Riemannian manifold $(\Omega_{\delta},g_{\rho})$ is quasi-isometric to $(\Omega, g_{\rho})$ with ratio $1+\delta$. Indeed, this follows from the fact that the intervals $[\delta^2,\delta]$ and $[0,\delta]$ are quasi-isometric with ratio $1+\delta$.

  At this stage, we have three Riemannian manifolds which are quasi-isometric with ratio close to $1$ : $(\Omega,g_{\Omega})$, $(\Omega,g_{\rho})$ and $(\Omega_{\delta},g_{\rho})$.

  To achieve (iii), we then consider $(\Omega,g_{\rho}) \times (\Sigma, g_{\Sigma})$ which is quasi-isometric to $(\Omega,g_{\Omega}) \times (\Sigma, g_{\Sigma})$ with ratio close to $1$. This implies that the two warped product metrics
  $g_{\Omega}\oplus h^2 g_{\Sigma}$ and $g_{\rho}\oplus h^2 g_{\Sigma}$ are also quasi-isometric with ratio close to $1$ for any $h$. The same property is true for
  $(\Omega,g_{\rho}) \times (\Sigma, g_{\Sigma})$ which is quasi-isometric to $(\Omega_{\delta},g_{\rho}) \times (\Sigma, g_{\Sigma})$.

  We can now construct the family of warping functions $h_{\delta}$ on $\Omega$: they take the value $1$ on $\partial \Omega$, the value $C$ on $\Omega_{\delta}$ and on $\Omega \setminus \Omega_{\delta}$, they satisfy $1\le h_{\delta} \le C$. In particular, as $h=C$ on $\Omega_{\delta}$, we observe that
  $(\Omega,g_{\rho})\times_C (\Sigma,g_{\Sigma})$ is quasi-isometric to  $(\Omega_{\delta},g_{\rho})\times_C (\Sigma,g_{\Sigma})$.

  We fix an eigenvalue $\lambda$ of $\Sigma$ and we compare $\sigma_{\lambda,l}((\Omega,g_{\rho})\times_{h_{\delta}} \Sigma)$ with $\sigma_{\lambda,l}((\Omega_{\delta},g_{\rho})\times_{h_{\delta}} \Sigma)$.
  Let $f,\tilde{f}$ be normalised eigenfunctions for $\sigma_{\lambda,i}((\Omega,g_{\rho})\times_{h_{\delta}} \Sigma)$, $\sigma_{\lambda,j}((\Omega,g_{\rho})\times_{h_{\delta}} \Sigma)$ respectively such that
$$
\int_{\partial \Omega } f\tilde{f}=\delta_{ij}
$$
and consider the restrictions $f_\delta,\tilde{f}_{\delta}$ of $f, \tilde{f}$ to $\Omega_{\delta}$. 
Since $T_{\delta}=\{p\in \Omega: d_{g_{\Omega}}(p,\partial \Omega)< \delta\}$, we have that $(T_{\delta},g_{\rho})$ is isometric to the product $(\partial \Omega,g_{\rho})\times [0,\delta]$. In the neighbourhood $T_\delta$, we can apply Lemma \ref{lemma:deBrunobis} with $t=0$ and $\eps=\delta$ to obtain
\begin{align*}
&\bigg\vert \int_{\partial \Omega_\delta} f_{\delta}(x)\tilde{f}_{\delta}(x)dx-\int_{\partial \Omega} f(x)\tilde{f}(x)dx \bigg\vert \\ 
& \qquad \qquad \le 
\vert \delta\vert^{1/2} \left(\left(\int_{[0,\delta]\times \partial \Omega} |df|^{2} \right)^{1/2}\left(\int_{\partial \Omega} \tilde{f}_{\delta}^{2}\right)^{1/2} \right. \nonumber\\
& \qquad \qquad \qquad \qquad \qquad \qquad \left.+ \left(\int_{[0,\delta]\times \partial \Omega}|d\tilde{f}|^{2} \right)^{1/2}\left(\int_{\partial \Omega} f^{2}\right)^{1/2}  \right)\\
& \qquad \qquad \le 
\vert \delta\vert^{1/2} \left(\sigma_{\lambda,i}((\Omega,g_{\rho})\times_{h_{\delta}} \Sigma)^{1/2}\left(\int_{\partial \Omega} \tilde{f}_{\delta}^{2}\right)^{1/2} \right. \nonumber \\
& \qquad \qquad \qquad \qquad \qquad \qquad \left. + \left(\sigma_{\lambda,j}((\Omega,g_{\rho})\times_{h_{\delta}} \Sigma) \right)^{1/2}\left(\int_{\partial \Omega} f^{2}\right)^{1/2}  \right).
\end{align*}
Now $\int_{\partial \Omega} f^{2} = 1$ and, by Lemma \ref{lemma:deBruno},
\begin{equation*}
       \bigl|\|\tilde{f}_\delta\|_{L^2(\partial \Omega)}-\|\tilde{f}\|_{L^2(\partial \Omega)}\bigr| \leq \delta^{1/2}\left(\int_{[0,\delta]\times \partial \Omega}|d f|^2\,dV\right)^{1/2},
\end{equation*}
so $\|\tilde{f}_\delta\|_{L^2(\partial \Omega)} \to \|\tilde{f}\|_{L^2(\partial \Omega)}$ as $\rho \to 0$.
Therefore
$$
\int_{\partial \Omega_{\delta} } f_\delta \tilde{f}_\delta =\delta_{ij}+o(1),
$$
as $\rho\to 0$.

Moreover
$$
\int_{\Omega_{\delta}}  (\vert df \vert^2h_{\delta}^n+\lambda f^2 h_{\delta}^{n-2})dV_{\Omega} < \int_{\Omega}  (\vert df \vert^2h_{\delta}^n+\lambda f^2 h_{\delta}^{n-2})dV_{\Omega}
$$
and the same is also true for $\tilde{f}$.

Hence, by the min-max characterisation, we deduce that, for a fixed $k$ and for $l\le k$, $$\sigma_{\lambda,l}((\Omega_{\delta},g_{\rho})\times_{h_{\delta}} \Sigma)\le \sigma_{\lambda,l}((\Omega,g_{\rho})\times_{h_{\delta}} \Sigma)+o(1)$$ as $\rho \to 0$.

As $(\Omega_{\delta},g_{\rho})\times_{h_{\delta}} \Sigma =(\Omega_{\delta},g_{\rho})\times_{C} \Sigma$ and $(\Omega_{\delta},g_{\rho})\times_{C} \Sigma$ is quasi-isometric to $(\Omega,g_{\rho})\times_{C} \Sigma$ with ratio close to $1$, we have the conclusion.
\end{proof}

\medskip
\noindent
Identity \eqref{Case n>2} was already proved in Section \ref{Spectrum $M_C$}.

\subsection{Proof of Theorem \ref{thm:Stability}}
In this section, we present a quantitative improvement of inequality~\eqref{ineq:introgeneral}, from which Theorem \ref{thm:Stability} will follow.

\begin{proposition}\label{prop:stablebody}
    Let $D\subset\Omega$ be a subdomain and consider a continuous function $f$ that is supported in $D$ with $f\geq 0$ everywhere and $f>0$ in the interior of $D$. Then,
    \begin{align*}
        \sigma_k(M_h)\leq-\frac{\lambda_k^2(\int_Dfh^{n-2}\,dV_{\Omega})^2}{\int\limits_D (|df|^2h^n+\lambda_k f^2h^{n-2})\,dV_{\Omega}}{\bigg(\int\limits_{\partial\Omega}h^n\,dA_{\partial\Omega}\bigg)^{-1}}\,+\lambda_k\frac{\int\limits_{\Omega}h^{n-2}\,dV_{\Omega}}{\int\limits_{\partial\Omega}h^n\,dA_{\partial\Omega}}.
    \end{align*}
\end{proposition}

\begin{proof}
     For $t\in\R$, we have
    \begin{align*}
    &R_{\lambda_k,h}(1-tf)=\frac{\int\limits_\Omega (t^2|df|^2h^n+\lambda_k (1-tf)^2h^{n-2})\,dV_{\Omega}}{\int\limits_{\partial\Omega}h^n\,dA_{\partial\Omega}}\\
    &=
    \frac{t^2(\int\limits_D (|df|^2h^n+\lambda_k f^2h^{n-2})\,dV_{\Omega}) - 2\lambda_k t\int\limits_D fh^{n-2}\,dV_{\Omega}}{\int\limits_{\partial\Omega}h^n\,dA_{\partial\Omega}}\,+\lambda_k\frac{\int\limits_{\Omega}h^{n-2}\,dV_{\Omega}}{\int\limits_{\partial\Omega}h^n\,dA_{\partial\Omega}},
    \end{align*}
    where we used that $f=0$ on $\partial \Omega$.
    This expression is quadratic in $t$, its minimum is achieved when
    $$t=t_0:=\frac{\lambda_k\int_D fh^{n-2}\,dV_{\Omega}}{\int\limits_D (|df|^2h^n+\lambda_k f^2h^{n-2})\,dV_{\Omega}}.$$
    Whence, the minimal value of $R_{\lambda_k,h}(1 - tf)$ leads to the following improvement of inequality~\eqref{ineq:introgeneral}:
\begin{align*}
\sigma_k(M_h)&\leq \sigma_{\lambda_k,0}(h) \leq R_{\lambda_k,h}(1 - t_0f)\\
&=-\frac{\lambda_k^2(\int_Dfh^{n-2}\,dV_{\Omega})^2}{\int\limits_D (|df|^2h^n+\lambda_k f^2h^{n-2})\,dV_{\Omega}}{\left(\int\limits_{\partial\Omega}h^n\,dA_{\partial\Omega}\right)^{-1}}\,+\lambda_k\frac{\int\limits_{\Omega}h^{n-2}\,dV_{\Omega}}{\int\limits_{\partial\Omega}h^n\,dA_{\partial\Omega}}.
\end{align*}
\end{proof}    

To prove Theorem \ref{thm:Stability} we consider the case where $h\equiv 1$ on $\partial\Omega$ and fibers of dimension $n=2$:
\begin{equation}\label{ineq:withdeficit}
\sigma_k(M_h)\leq \frac{-\lambda_k^2(\int_Df\,dV_{\Omega})^2}{|\partial\Omega|\int\limits_D (|df|^2h^2+\lambda_k f^2)\,dV_{\Omega}}\,+\lambda_k\frac{|\Omega|}{|\partial\Omega|}.    
\end{equation}

Let $\delta:= \frac{\lambda_k \vert \Omega\vert}{\vert \partial \Omega\vert}-\sigma_k(M_h)>0$.
If $D$ is a  ball $B(q,r)\subset \Omega$, then let $f:D\to\R$ be defined by $f(x)=d(x,\partial D)$. We can then compute $|df|=1$ and $f\leq r$, and also $f\geq r/2$ on $B(q,\frac{r}{2})$ so that
\begin{align*}
\sigma_k(M_h)\leq \frac{-\lambda_k^2r^2 |B(q,\frac{r}{2})|}{4|\partial\Omega|(\int\limits_D h^2+\lambda_k r^2|D|)}\,+\lambda_k\frac{|\Omega|}{|\partial\Omega|}.
\end{align*} 
It follows that
$$
\frac{\lambda_k^2r^2 |B(q,\frac{r}{2})|}{4|\partial\Omega|(\int\limits_D h^2+\lambda_k r^2|D|)} \le \delta,
$$
and
$$
\int_{D}h^2dV_{\Omega}\ge \frac{\lambda_kr^2}{4 \delta}\frac{\vert B(q,\frac{r}{2})\vert}{\vert \partial \Omega\vert}-\lambda_kr^2 \vert D\vert
$$
as required.

\noindent
We  then deduce that \eqref{eq:limstab} follows concluding the proof of Theorem \ref{thm:Stability}.

\subsection{Upper bounds in terms of $L^p$-norms.}

In this section, we prove Corollary \ref{coro: integralbound}, Theorem \ref{thm:3dimconnected}, and Theorem \ref{dim basis 1} which share the constraint that the $L^p$-norm of $h$ is bounded.

\begin{proof}[Proof of Corollary \ref{coro: integralbound}.]
Since $h=1$ on $\partial \Omega$, Inequality \eqref{ineq:introgeneral} becomes
$$ \sigma_k(M_h) < \frac{\lambda_k}{\vert \partial \Omega \vert} \int_{\Omega} h^{n-2} dV_\Omega.$$
Using that $p \geq n-2 \geq 1$ and $|\Omega|<\infty$,  H\"{o}lder's inequality with  conjugate exponents $\frac{p}{n-2}$ and $\frac{p}{p - n + 2}$ gives
\begin{equation*}
    \int_{\Omega} h^{n-2} dV_{\Omega} 
    \leq \Vert h \Vert_{L^p}^{n-2} \vert \Omega \vert^{(p-(n-2))/p}.
\end{equation*}
Hence 
\begin{equation*}
    \sigma_k(M_h) < \frac{\lambda_k \Vert h\Vert_{L^p}^{n-2} \vert \Omega \vert^{(p-(n-2))/p}}{\vert \partial \Omega \vert}.
\end{equation*}
\end{proof}

Next, we prove Theorem \ref{thm:3dimconnected}.
\begin{proof}[Proof of Theorem \ref{thm:3dimconnected}.]
By Lemma \ref{lemma:quasi-iso},
up to a quasi-isometry, for $\eps$ sufficiently small, we can suppose that the neighbourhood 
$$
\{p \in \Omega: d(p,\partial \Omega) \le 2\epsilon \}
$$
is isometric to $\partial \Omega \times [0,2\epsilon]$.  
We have $\vert \partial \Omega \times [\eps,2\eps]\vert =\eps \vert \partial \Omega\vert.$

We choose $h_{\eps}=\left(\frac{C}{2 \eps \vert \partial \Omega\vert}\right)^{1/p}$ on $[\eps, 2\eps]$ so that $\int_{\partial \Omega \times [\eps,2\eps]}h_{\eps}^p= \frac{C}{2}$, and extend $h_{\eps}$ so that :
(i) $\int_{\Omega} h_{\eps}^p \leq C$ and (ii) $h_{\eps}\ge D$ on $[0,2\eps]$. Here, $D$ is just a positive constant: we have to avoid that $h_{\eps} \to 0$ in the interval $[0,2\eps]$.

The first eigenvalue, $\sigma_1(M_{h_{\eps}})$, is the minimum between $\sigma_{0,1}(h_\eps)$ and $\sigma_{\lambda_{1,0}}(h_\eps)$. We consider these two cases separately, and we aim to show that in each case $\sigma_1(M_{h_\eps}) \to \infty$ as $\eps \to 0$. 

\noindent
\textbf{Case of $\sigma_{\lambda_1,0}$}.
Let $a$ be an eigenfunction corresponding to $\sigma_{\lambda_1,0}$ such that $\int_{\partial \Omega} a^2=1$.
We consider this function $a$ on $\partial \Omega \times [\eps,2 \eps]$. 

\noindent
\textbf{First case:} If for each $t\in [\eps,2 \eps]$, $\int_{\partial \Omega \times \{t\}} a^2(x,t) \ge \frac{1}{2}$, then, because $n-2 > p$ and $\lambda_1(\Sigma)>0$, we have
\begin{align*}
\int_{\partial \Omega \times [\eps,2\eps]}\lambda_1(\Sigma) a^2 h^{n-2} 
&\ge \lambda_1(\Sigma)\frac{1}{2} \eps h^{n-2}_{\eps} \\
&= \frac{\lambda_1(\Sigma)}{2} \left(\frac{C}{2 \vert \partial \Omega\vert}\right)^{(n-2)/p} \eps^{(p - n + 2)/p} \to \infty
\end{align*}
as $\eps \to 0$.

\noindent
\textbf{Second case:} In this case, there exists $t\in [\eps,2 \eps]$ such that $\int_{\partial \Omega \times \{t\}} a^2(x,t) < \frac{1}{2}$. 
We use \eqref{eq:lem2.1} from the proof of Lemma \ref{lemma:deBruno} with $f=a$ to obtain:

$$
\bigg\vert \int_{\partial \Omega} a^2(x,t)dA_{\partial \Omega} -\int_{\partial \Omega} a^2(x,0)dA_{\partial \Omega} \bigg \vert \le \vert t\vert^{1/2}\left(\int_{\partial \Omega \times [0,t]} \vert da(x,s)\vert^2 dV_{\Omega}\right)^{1/2},
$$
and, as $\int_{\partial \Omega} a^2(x,0)dA_{\partial \Omega} = 1$, this implies
$$
\int_{\partial \Omega \times [0,t]} \vert da(x,s)\vert^2 dV_{\Omega}\ge \frac{1}{8\eps}
$$
so
$$
\int_{\partial \Omega \times [0,t]} \vert da(x,s)\vert^2 h_\eps^n dV_{\Omega}\ge \frac{D^n}{8\eps}
$$
which tends to $\infty$ as $\eps \to 0$.

\noindent
\textbf{Case of $\sigma_{0,1}$}. 
Let $a$ be an eigenfunction corresponding to $\sigma_{0,1}$. 
We observe that the eigenfunction $a$ has to be orthogonal on $\partial \Omega$ to the eigenfunction for $\sigma_{0,0}$, that is the constant function. So, we have
$$
\int_{\partial \Omega} a^2(x,0)dA_{\partial \Omega} =1, \quad  
\int_{\partial \Omega} a(x,0)dA_{\partial \Omega} =0.
$$
We consider the function $a$ in the product
$\partial \Omega \times [\eps,2\eps]$.

\noindent
\textbf{First case:} Suppose that $\int_{\partial \Omega \times [\eps,2\eps]}\vert da\vert^2 dV_{\Omega} \ge 1$. Then, as before, 
$$\int_{\partial \Omega \times [\eps,2\eps]}\vert da\vert^2 h_\eps^n dV_{\Omega} \geq \left(\frac{C}{2\eps \vert \partial \Omega \vert}\right)^{n/p} \to \infty $$
as $\eps \to 0$.

\noindent
\textbf{Second case:} Suppose that $\int_{\partial \Omega \times [\eps,2\eps]}\vert da\vert^2 dV_{\Omega} \le 1$.

If $\int_{\partial \Omega \times [0,\eps]}\vert da\vert^2 dV_{\Omega} \ge \frac{1}{\eps^{1/2}}$, then, since $h_\epsilon \geq D$ on $[0,\eps]$, we would have that 
$$\sigma_{0,1}(M,h_{\eps}) \geq \int_{\partial \Omega \times [0,\eps]}\vert da\vert^2 h_\eps^n dV_{\Omega}
\geq \frac{D^n}{\eps^{1/2}}\to \infty $$ as $\eps \to 0$.

So it remains to consider the case where $\int_{\partial \Omega \times [0,\eps]}\vert da\vert^2 dV_{\Omega} \le \frac{1}{\eps^{1/2}}$.
For each $t \in [\eps, 2\eps]$, by Lemma \ref{lemma:deBruno}, we have that
\begin{align*}
\bigl|\|a(\cdot, t)\|_{L^2(\partial \Omega)} - \|a(\cdot, 0)\|_{L^2(\partial \Omega)}\bigr|^2
   & \leq t \int_{[0,t]\times \partial \Omega}|d a|^2\,dV_{\Omega} \\
   & \leq 2\eps \int_{[0,2\eps]\times \partial \Omega}|d a|^2\,dV_{\Omega} \\
   & \leq 2\eps^{1/2}(1+\eps^{1/2}).
   \end{align*}
Hence $\int_{\partial \Omega } a^2(x,t)dA_{\partial \Omega}=1+o(1)$, as $\eps\to 0$.
In addition, by Lemma \ref{lemma:deBrunobis} with $f=a$ and $\tilde{f}=1$, we have that
\begin{align*}
\bigg\vert \int_{\partial \Omega} a(x,t) dA_{\partial \Omega} -\int_{\partial \Omega} a(x,0) dA_{\partial \Omega} \bigg\vert 
&\le t^{1/2} |\partial \Omega|^{1/2} \left(\int_{[0,t]\times \partial \Omega} |da|^{2} dV_{\Omega}\right)^{1/2}\\
&\le (2\eps)^{1/2} |\partial \Omega|^{1/2} \left(\int_{[0,2\eps]\times \partial \Omega} |da|^{2} dV_{\Omega}\right)^{1/2}\\
&\leq |\partial \Omega|^{1/2} (2\eps^{1/2}(1 + 2\eps^{1/2}))^{1/2}.
\end{align*}

\noindent
Therefore $\int_{\partial \Omega } a(x,t)dA_{\partial \Omega}=o(1)$, as $\eps\to 0$.

\noindent
Thus we deduce that
$\int_{\partial \Omega \times \{t\}}\vert da\vert^2dA_{\partial \Omega}\ge (\lambda_1(\partial \Omega)+o(1))\int_{\partial \Omega \times \{t\} } a^2(x,t)dA_{\partial \Omega}$
and
$$
\int_{\partial \Omega \times [\eps,2\eps] }\vert da\vert^2dV_{\Omega}\ge (\lambda_1(\partial \Omega)+o(1))\int_{\partial \Omega \times [\eps,2\eps] } a^2(x,t)dV_{\Omega} = O(\eps).
$$
As $\partial \Omega$ is connected, we have $\lambda_1(\partial \Omega)>0$, and this implies $$\int_{\partial \Omega \times [\eps,2\eps] } \vert da\vert^2h^n_{\eps}dV_{\Omega} \to \infty$$ as $\eps \to 0$.
\end{proof}

\begin{rem}
    By adapting the above proof taking ${h_{\eps}=\frac{1}{\eps^{(n-1)/(n-2)}}}$ on $[\eps,2\eps]$ instead, it follows that if $n \geq 3$ and $\partial \Omega$ is connected, then there exists a sequence of functions $(h_\eps)_\eps$, that tend to infinity close to the boundary of $\Omega$, such that $\sigma_k(M_{h_\eps}) \to \infty$ as $\eps \to 0$.
\end{rem}

\noindent
Finally, we prove Theorem \ref{dim basis 1}.
\begin{proof}[Proof of Theorem \ref{dim basis 1}.]

The function $h$ is uniformly continuous in the interval $[0,L]$, and there exists $\delta>0$ such that $\vert x-y\vert \le \delta$ implies $\vert h(x)-h(y)\vert \le (C/L)^{1/p}.$

Let an even integer $N$ be such that $\frac{L}{N}\le \delta\le \frac{L}{N-1}$.
We introduce the partition of the interval $[0,L]$ given by $0< \frac{L}{N}<...<\frac{kL}{N}<\frac{(k+1)L}{N}<...<\frac{NL}{N}=L$. By hypothesis, if $\frac{kL}{N}\le x,y \le \frac{(k+1)L}{N}$, $\vert h(x)-h(y)\vert \le (C/L)^{1/p}$.
We deduce that, for $\frac{kL}{N}\le x,y \le \frac{(k+1)L}{N}$, if $h(x)\ge 3(C/L)^{1/p}$, then $h(y)\ge 2(C/L)^{1/p}$.

Let $M$ be the number of intervals $[\frac{kL}{N},\frac{(k+1)L}{N}]$ that contain a point $x$ with $h(x)\ge 3(C/L)^{1/p}$. In such intervals, we have for all points $y$ that $h(y)\ge 2(C/L)^{1/p}$. Hence, since $p \geq 1$, $2^p \geq 2$ and we get 
$$\int_{\frac{kL}{N}}^{\frac{(k+1)L}{N}} h^p\ge \frac{2C}{N}.$$ 
If we have $M$ intervals of this type, then
$$\int_{0}^{L}h^p \ge \frac{2CM}{N}. $$ 
As $\int_{0}^{L}h^p \leq C$, we deduce that $M\le \frac{N}{2}$.

We now consider the test function $a$, with $a(0)=-1$ and $a(L)=1$ defined as follows: we select $N/2$ intervals where $h\le 3(C/L)^{1/p}$. This test function will vary only on these intervals, increasing with a slope of $\frac{4}{L}$ and will be constant in all the other intervals, in order to be continuous. In each of the chosen intervals, $a'(x)=\frac{4}{L}$ and $h\le 3(C/L)^{1/p}$ so that $$\int_0^La'(x)^2h^n(x) \le \frac{16}{L^2}\frac{L}{2}(3(C/L)^{1/p})^n=\frac{8 \cdot 3^n C^{n/p}}{L^{\frac{n+p}{p}}},$$
which gives that
$$\sigma_1(M_L) \leq \frac{4 \cdot 3^n C^{n/p}}{L^{\frac{n+p}{p}}}.$$
For $k \geq 1$, using \eqref{ineq:kcompk0} and taking $a=1$ as a test function in the variational characterisation for $\sigma_{\lambda_k,0}(h)$, we obtain
\begin{equation*}
    \sigma_k(M_L) \leq \sigma_{\lambda_k,0}(h) \leq \frac{1}{2}\int_0^L \lambda_k h^{n-2}
    \leq \frac{3^{n-2}C^{(n-2)/p}}{4L^{\frac{n-2-p}{p}}}\lambda_k.
\end{equation*}
\end{proof}

\bibliographystyle{plain}              
\bibliography{bibliography}

\end{document}